\documentclass[12pt]{amsart}
\usepackage{amssymb}
\usepackage{tikz-cd}
\numberwithin{equation}{section}

\def\O{{\mathcal{O}}}
\def\p{{\mathfrak{p}}}

\def\C{\mathbb{C}}
\def\Z{\mathbb{Z}}
\def\Q{\mathbb{Q}}
\def\R{\mathbb{R}}

\def\Z{{\mathbb Z}}

\def\O{{\mathcal{O}}}

\newcommand\mum{\boldsymbol\mu_{m^\infty}}
\newcommand\mup{\boldsymbol\mu_{p^\infty}}
\newcommand\mud{\boldsymbol\mu_{d^\infty}}
\newcommand\mutwo{\boldsymbol\mu_{2^\infty}}

\newenvironment{Gal}{{\rm Gal\,}}{}

\newtheorem{theorem}{Theorem}[section]
\newtheorem{lemma}[theorem]{Lemma}

\newtheorem{proposition-definition}[theorem]{Proposition-Definition}
\newtheorem{corollary}[theorem]{Corollary}
\newtheorem{conjecture}[theorem]{Conjecture}

\theoremstyle{definition}
\newtheorem{example}[theorem]{Example}
\newtheorem{question}[theorem]{Question}

\theoremstyle{remark}
\newtheorem*{remark}{Remark}

\makeatletter
\@namedef{subjclassname@2020}{\textup{2020} Mathematics Subject Classification}
\makeatother

\title{Roots of unity and higher ramification in iterated extensions}
\author[S. Hamblen]{Spencer Hamblen}
\address{Department of Mathematics and Computer Science, McDaniel College, 2 College Hill, Westminster, MD, 21157, USA}
\email{shamblen@mcdaniel.edu}
\author[R. Jones]{Rafe Jones}
\address{Department of Mathematics and Statistics, Carleton College, 1 North College St, Northfield, MN, 55057, USA}
\email{rfjones@carleton.edu}
\thanks{The first author's research was partially supported by the McDaniel College Faculty Development Fund.  The second author's research was partially supported by an AMS-Simons Research Enhancement Grant for PUI faculty.}

\begin{document}

\begin{abstract}
Given a field $K$, a rational function $\phi \in K(x)$, and a point $b \in \mathbb{P}^1(K)$, we study the extension $K(\phi^{-\infty}(b))$ generated by the union over $n$ of all solutions to $\phi^n(x) = b$, where $\phi^n$ is the $n$th iterate of $\phi$. We ask when a finite extension of $K(\phi^{-\infty}(b))$ can contain all $m$-power roots of unity for some $m \geq 2$, and prove that this occurs for several families of rational functions. A motivating application is to understand the higher ramification filtration when $K$ is a finite extension of $\Q_p$ and $p$ divides the degree of $\phi$, especially when $\phi$ is post-critically finite (PCF). We show that all higher ramification groups are infinite for new families of iterated extensions, for example those given by bicritical rational functions with periodic critical points. We also give new examples of iterated extensions with subextensions satisfying an even stronger ramification-theoretic condition called arithmetic profiniteness. We conjecture that every iterated extension arising from a PCF map should have a subextension with this stronger property, which would give a dynamical analogue of Sen's theorem for PCF maps. 
\end{abstract}

\maketitle

\section{Introduction}

The Galois theory of iterates of a rational function, also called the study of arboreal Galois representations, has attracted recent interest for a number of reasons. Among them is the search for large, explicitly given Galois extensions of $\Q$, and several authors have given conditions under which the Galois groups in this setting are as large as possible (see \cite{survey} for a survey). Such results rely on a study of ramification over various primes of the ground field, but this ramification is almost always tame, and any wild ramification plays little role. Thus far the wild case has not been studied in detail, with two notable exceptions: \cite{Poonen2}, which treats unicritical polynomials and focuses on the size of the Galois groups rather than the depth of ramification, and \cite{sing1, sing2}, which give remarkably detailed results on the higher ramification filtration in certain circumstances. 

Here our aim is to study the depth of wild ramification in iterated $p$-adic extensions for new classes of rational functions. To this end, we look for $p$-power roots of unity in iterated extensions of an arbitrary field. More specifically, for a field $K$ and a rational function $\phi \in K(x)$, denote by $\phi^{-n}(b)$ the set of all solutions (in a fixed algebraic closure $\overline{K}$) to $\phi^n(x) = b$, where $b \in \mathbb{P}^1(K)$ and $\phi^n$ is the $n$th iterate of $\phi$. Take $\phi^{-\infty}(b)$ to be the full backward orbit of $b$, that is, $\bigcup_{n \geq 0} \phi^{-n}(b)$, where $\phi^0(x) = x$. 
We consider the following question: 
\begin{question} \label{rootsquest}
Let $K$ be a field. For which $\phi \in K(x)$ and $b \in \mathbb{P}^1(K)$ does there exist $m \geq 2$ such that $K(\mum)$ is contained in a finite extension of $K(\phi^{-\infty}(b))$?
\end{question}

Here $\mum$ is the union of all roots in $\overline{K}$ of $x^{m^k} - 1$, as $k$ varies over the positive integers. Our main result addressing Question \ref{rootsquest} is the following:

\begin{theorem} \label{mainroots}
Let $K$ be a field and let $\phi \in K(x)$ have degree $d \geq 2$, where $d$ is not divisible by the characteristic of $K$. Then $K(\mum)$ is contained in a finite extension of $K(\phi^{-\infty}(b))$ for any non-exceptional $b$ if:
\begin{enumerate}
    \item $\phi$ is M\"obius-conjugate (over $\overline{K}$) to a rational function in $x^m$ for which $0$ and $\infty$ are periodic;
    \item $\phi$ is M\"obius-conjugate (over $\overline{K}$) to $T_d$, the degree-$d$ monic Chebyshev polynomial, where $m \mid d$;
    \item $\phi$ is a flexible Latt\`es map of degree divisible by $m$.
\end{enumerate}
\end{theorem}
Recall that $b$ is exceptional (for $\phi$) if $\phi^{-\infty}(b)$ is a finite set. For more precise versions of the three statements given in Theorem \ref{mainroots}, see Corollary \ref{conjunity}, Theorem \ref{chebmain}, and Theorem \ref{Lattesroots}. Note that if the conjugacy in part (1) is defined over $K$, then in fact $K(\mum) \subseteq K(\phi^{-\infty}(b))$; see Theorem \ref{rootsof1} and Corollary \ref{conjunity}. As an application of part (1) of Theorem \ref{mainroots}, we show that if $\phi$ has precisely two critical points, both of which are periodic, then $K(\mud)$ is contained in a finite extension of $K(\phi^{-\infty}(b))$ for any non-exceptional $b$ (see Corollary \ref{twocritpts}). This generalizes prior work on unicritical polynomials in \cite[Section 3]{pagano} and \cite[Section 3]{ferraguti-pagano}.

Question \ref{rootsquest} asks when (a finite extension of) the field generated by $\phi^{-\infty}(b)$ can contain all of a certain class of special points. Recently there has been much work on the related question of determining when $\phi^{-\infty}(b)$ itself can contain special points. Andrews and Petsche \cite{petsche} ask under what conditions $\phi^{-\infty}(b) \subseteq K^\text{ab}$ when $K$ is a number field, $\phi \in K[x]$, and $K^\text{ab}$ is the maximal abelian extension of $K$. While we don't directly address this question, our Theorem \ref{rootsof1} gives conditions under which $K(\phi^{-\infty}(b)) \cap K^\text{ab}$ contains $K(\mum)$, and hence is an infinite extension of $K$. In a similar vein to the work of Andrews and Petsche, the authors of \cite{foz} consider when it is possible for a rational function $\phi \in K(x)$ to have $\phi^{-\infty}(b) \cap K^\text{cyc}$ infinite, where $K^\text{cyc}$ denotes the maximal cyclotomic extension of $K$. In \cite[Lemma 3.4]{foz} they show this can only occur when $\phi$ is a PCF map with all critical points strictly preperiodic, $\phi$ is conjugate to $x^{\pm d}$, or $\phi$ is conjugate to a Chebyshev polynomial.  

We remark that it might be interesting to study variants of Question \ref{rootsquest} for other classes of special points, such as $x$-coordinates of torsion points on an elliptic curve.

We use Theorem \ref{mainroots} to study ramification in $K(\phi^{-\infty}(b))$ when $K$ is a finite extension of $\Q_p$. Let $K$ be such an extension and $L$ a Galois extension of $K$ with $G = \Gal(L/K)$. Denote by $G^u$ the ramification group in the upper numbering corresponding to $u \in \mathbb{R}_{\geq 0}$ (see e.g. \cite[Section IV.3]{serrebook} or \cite{wildthings} for definitions). 
As a consequence of Theorem \ref{mainroots}, we obtain:
\begin{theorem} \label{mainram1}
Let $K$ be a finite extension of $\Q_p$, and assume that $\phi$ satisfies one of conditions (1)-(3) in Theorem \ref{mainroots} with $p \mid m$. Then for any non-exceptional $b \in \mathbb{P}^1(K)$, the groups $\Gal(K(\phi^{-\infty}(b))/K)^u$ are infinite for all $u \in \mathbb{R}^+$.
\end{theorem}

Hajir, Aitken, and Maire \cite{hajir} were among the first to address questions about wild ramification in iterated extensions, and they were particularly interested in constructing infinite extensions of a number field that are  ramified over only finitely many primes, and with only tame ramification. To that end, they ask \cite[Question 7.1]{hajir}, whether there is a number field $K$, a map $\phi \in K(x)$ of degree $d \geq 2$, and $b \in \mathbb{P}^1$ such that $K(\phi^{-\infty}(b))$ is ramified over only finitely many primes of $K$, and unramified at all primes of $K$ dividing $d$. Such a map must be post-critically finite (PCF) by \cite[Theorem 5]{finiteram}, and in the same paper it is conjectured that the answer to \cite[Question 7.1]{hajir} is no \cite[Conjecture 6]{finiteram}. Quite recently this conjecture was proved for PCF polynomials of prime-power degree (see \cite[Theorem 4.3]{sing1} and also \cite[Proposition 2.1]{sing1}, which shows the hypothesis of monomial reduction is satisfied for PCF polynomials). Our Theorem \ref{mainram1} applies to some PCF maps, and thus proves new cases of \cite[Conjecture 6]{finiteram}, as it allows non-polynomial maps and polynomials that do not have prime-power degree.

In contrast to the conclusion of Theorem \ref{mainram1}, wild ramification need not be very deep for iterated extensions in general; see \cite[Theorem 5.11 and Example 5.12]{Poonen2}, where the higher ramification groups are trivial for $u$ sufficiently large. Nonetheless, the presence of $p$-power roots of unity is not required to produce deep ramification, and it is an interesting problem to determine which pairs $(\phi, b)$ have $\Gal(K(\phi^{-\infty}(b))/K)^u$ infinite for all $u \in \mathbb{R}^+$. 

Our last main result uses work of Cais, Davis, and Lubin \cite{CDL} to give a class of examples satisfying a stronger condition than all higher ramification groups being infinite. Let $K$ be a finite extension of $\Q_p$ with algebraic closure $\overline{K}$ and $L$ an infinite, totally wildly ramified Galois extension of $K$ with Galois group $G$. Then $L/K$ is \textit{arithmetically profinite (APF)} if $G^u$ is an open subgroup of $G$ for all $u \in \mathbb{R}^+$. More generally, if $L \subseteq \overline{K}$ is any infinite, totally wildly ramified extension of $K$, we say it is APF if the product 
$$
\Gal(\overline{K}/K)^u\Gal(\overline{K}/L)
$$
is an open subgroup of $\Gal(\overline{K}/K)$, for all $u \in \mathbb{R}^+$. If $L/K$ is APF and $E$ is any Galois extension of $K$ containing $L$, then
$\Gal(E/K)^u$ is infinite for all $u \in \mathbb{R}^+$ (see p. \pageref{apfdisc} for further discussion). Finally, we say that $K(\phi^{-\infty}(b))$ is \textit{branch-APF} over $K$ if there is a sequence $\{\alpha_n\}_{n \geq 0}$ with $\alpha_0 = b$ and $\phi(\alpha_n) = \alpha_{n-1}$ for all $n \geq 1$, such that $K(\alpha_1, \alpha_2, \ldots)$ is APF over $K$. 

Our result is the following, where $\overline{\phi}$ denotes the coefficient-wise reduction of $\phi$ modulo the maximal ideal of $K$. 
\begin{theorem} \label{mainapfintro}
Let $K$ be a finite extension of $\Q_p$ and $\phi \in K(x)$. Suppose that there exist $m, r \geq 1$ and $c \neq 0$ with 
$
\overline{\phi^m}(x) = cx^{p^r} 
$
and $\deg \phi^m = p^r$. Then $\phi^m$ has a unique fixed point $\gamma$ (resp. $\delta$) with positive (resp. negative) $p$-adic valuation.  Additionally, if we let $K' = K(\gamma, \delta)$,
then there is a unit $w \in \overline{K}$ such that $K'(w)(\phi^{-\infty}(b))$ is branch-APF over $K'(w)$, for any $b$ of the form
$$
\frac{\pi + w\gamma}{\delta^{-1}\pi + w},
$$ 
where $\pi$ is a uniformizer of $K'(w)$.
\end{theorem}

We consider $\infty$ to have negative valuation, so $\delta = \infty$ is allowed, and in this case we take $\delta^{-1} = 0$.  We may take $w  = 1$ when $\delta = \infty$ and the leading coefficient of the numerator of $\phi^m$ coincides with the denominator of $\phi^m$ evaluated at $\gamma$. In particular, we may take $w = 1$ when $\phi$ is a monic polynomial. In \cite{sing1}, Sing gives a detailed description of the higher ramification filtration associated to extensions of the form $K(\alpha_1, \alpha_2, \ldots)$ for certain post-critically bounded $\phi$ and any basepoint $b$. Our Theorem \ref{mainapfintro} applies to more maps, but with restricted basepoints; see Example \ref{nonPCB}.
Determining whether it is possible for the full extension $K(\phi^{-\infty}(b))$ to be APF is a very interesting open problem. 

Let us return now to PCF maps. The authors of \cite{hajir} already speculated that a negative answer to their Question 7.1 was likely, observing that experimentation suggests that $K(\phi^{-\infty}(b))$ is deeply ramified -- that is, all higher ramification groups are infinite -- at primes of residue characteristic dividing $d$.  For the PCF maps to which Theorem \ref{mainram1} applies, we confirm the authors' speculation.  Motivated partly by Theorem \ref{mainapfintro}, we conjecture that a still stronger property holds for any PCF map over $\Q_p$ whose degree is divisible by $p$:
\begin{conjecture}[Arboreal Sen Conjecture for PCF maps] \label{arbsen}
Let $K$ be a finite extension of $\Q_p$ and $\phi \in K(x)$ a PCF rational function of degree $d \geq 2$. If $p \mid d$ and $b \in \mathbb{P}^1(K)$ is not exceptional for $\phi$, then $K(\phi^{-\infty}(b))$ is a branch-APF extension of $K$.
\end{conjecture}

Sen's celebrated theorem \cite{sen} shows that a Galois extension of local fields of characteristic zero with $p$-adic analytic Galois group $G$ must be APF. Indeed, Sen shows that the higher ramification filtration on $G$ is tightly connected to the natural filtration on $G$ arising from it being a $p$-adic Lie group. As noted above, a direct translation of Sen's theorem to the setting of general iterated extensions does not hold (\cite[Example 5.12]{Poonen2}). However, ramification in iterated extensions is closely linked to the forward orbits of the critical points, and heuristically the periodicity of the critical points of a PCF map suggests a compounding of ramification that should lead to branch-APF extensions. Moreover, evidence for Conjecture \ref{arbsen} comes from work of Sing \cite[Theorem 4.3, Proposition 2.1]{sing1}, who proves the conjecture for PCF polynomials of prime power degree that satisfy a technical hypothesis ($p \nmid d$ in the notation of \cite{sing1}). Sing proves her result in much the way that Sen did: by showing that the higher ramification filtration and a natural filtration arising from iteration are closely linked. 

An outline of the paper is as follows. In Section \ref{unity} we study roots of unity in iterated extensions and prove Theorem \ref{mainroots}, in Section \ref{preliminaries} we prove Theorem \ref{mainram1}, and in Section \ref{apf} we prove Theorem \ref{mainapfintro}.

\section{Roots of unity in iterated extensions} \label{unity}

In this section we study Question \ref{rootsquest}, and give a partial answer in the form of a proof of Theorem \ref{mainroots}. Theorem \ref{mainroots} shows that Question \ref{rootsquest} has a positive answer for the usual suspects: when $\phi$ is a power map, Chebyshev polynomial, or Latt\`es map, and $b$ is any basepoint save for the necessarily excluded exceptional case.  However, there is at least one other class of $\phi$ for which Question \ref{rootsquest} has a positive answer, namely certain maps that factor through a power map. Witness two motivating examples, which come in the work of Pink \cite{pink2} and Ahmed et al. \cite{arithbas}. Pink studies the situation where $k$ is any field of characteristic not equal to 2, $t$ is transcendental over $k$, $K= k(t)$, and $b = t$. If $\phi(x) = x^2 + c \in k[x]$ has $\phi^n(0) = 0$ for some $n \geq 1$, then \cite[Theorem 2.8.4]{pink2} implies that $k(\mutwo)$ is contained in $K(\phi^{-\infty}(b))$. Ahmed et al. study the setting where $K = \Q$, $\phi(x) = x^2 - 1$, and $b \in \Q \setminus \{-1,0\}$, and they show that $\Q(\mutwo)$ is contained in $K(\phi^{-\infty}(b))$ \cite[Lemma 1.4]{arithbas}.  

We begin by generalizing these results in Subsection \ref{power}, and in the process shed some light on why periodicity of the critical point is a vital assumption. In Subsection \ref{special}, we treat the case of Chebyshev polynomials and Latt\`es maps. One of the strengths of the results in this section is that they hold for all non-exceptional choices of $b$. 

A final observation before we get to the proofs: an affirmative answer to Question \ref{rootsquest} is equivalent to the degree  $[K(\mum) : K(\phi^{-\infty}(b)) \cap K(\mum)]$ being finite.  Indeed, it is clear that $K(\mum)$ is contained in a finite extension of $K(\phi^{-\infty}(b))$ if and only if $F:= K(\mum)K(\phi^{-\infty}(b))$ is a finite extension of $K(\phi^{-\infty}(b))$.  But $\Gal(F/K(\phi^{-\infty}(b)))$ is isomorphic to $\Gal(K(\mum) / (K(\phi^{-\infty}(b)) \cap K(\mum)))$, showing the non-trivial direction of the equivalence.

\subsection{Maps that factor through a power map} \label{power}

\begin{theorem} \label{rootsof1}
Let $K$ be a field and $\phi \in K(x)$ a rational function of degree $d \geq 2$, where $d$ is not divisible by the characteristic of $K$. Suppose that $\phi \in K(x^m)$ for some integer $m \geq 2$ (necessarily with $m \mid d)$. Assume further that $\phi(0) = 0$ and $\phi(\infty) = \infty$. 
Then for any $b \in \mathbb{P}^1(K)$, we have $K(\mum) \subseteq K(\phi^{-\infty}(b))$ unless $b \in \{0, \infty\}$ and $\phi^{-1}(b) = \{b\}$.
\end{theorem}

\begin{proof}
Let $f,g \in K[x]$ be relatively prime polynomials with $f$ monic and $\phi = f/g$. The hypotheses of the theorem give that
\begin{enumerate}
    \item $f, g \in K[x^m]$;
    \item $f(0) = 0$;
    \item $\deg f > \deg g$.
\end{enumerate}
Let $B = \bigcup_{n \geq 1} \phi^{-n}(b)$ be the backwards orbit of $b$ under $\phi$, and observe that $K(\phi^{-\infty}(b)) = K(B)$. Put
$$R = \left\{\frac{\alpha_1 \alpha_2 \cdots \alpha_j}{\beta_1\beta_2 \cdots \beta_j} : \alpha_i, \beta_i \in B \setminus \{0, \infty\} \right\}.$$
Note that to be in $R$, the product in the numerator must have the same number of terms as the product in the denominator. We claim that every element of $R$ has an $m$th root that is also in $R$. Let $\alpha \in B \setminus \{0, \infty\}$, and note that $B \setminus \{0, \infty\}$ contains the solutions to $\phi(x) - \alpha = 0$, and hence contains the roots of $h(x) := f(x) - \alpha g(x)$. Letting $g(0) = c$, we have by hypothesis that $h(0) = f(0) - \alpha g(0) = -c \alpha$. Note that $c \neq 0$ because $f(0) = 0$ and $f$ and $g$ are relatively prime. Moreover, $h$ must be monic, because $f$ is monic and $\deg f > \deg g$. It follows that 
\begin{equation} \label{prod1}
\prod_{\gamma : h(\gamma) = 0} \gamma = (-1)^d h(0) = (-1)^{d+1} c \alpha.
\end{equation}

Because $K$ has characteristic not dividing $d$ and hence not dividing $m$, the polynomial $x^m - 1$ has $m$ distinct roots in $\overline{K}$, and there is $\zeta \in \overline{K}$ with $\zeta^m = 1$ and $\zeta^i \neq 1$ for all $i < m$. 
By assumption $h \in K[x^m]$, and so the roots of $h$ are invariant under multiplication by $\zeta$. Indeed, multiplication by $\zeta$ gives an action of the cyclic group of order $m$ on the roots of $h$, and each orbit of this action has $m$ elements. Let $\gamma_1, \ldots, \gamma_{d/m}$ be a choice of one element from each orbit. For each $i$, the product of the elements in the orbit of $\gamma_i$ is
$$
\prod_{j = 1}^m \zeta^j \gamma_i = \zeta^{m(m-1)/2} \gamma_i^m = (-1)^{m+1}\gamma_i^m.
$$
Therefore
$$
\prod_{\gamma : h(\gamma) = 0} \gamma = (-1)^{d(m+1)/m}\left( \prod_{i=1}^{d/m} \gamma_i \right)^m,
$$
and \eqref{prod1} then gives
\begin{equation} \label{prod2}
(-1)^{(m+d)/m}c\alpha =  \left( \prod_{i=1}^{d/m} \gamma_i \right)^m,
\end{equation}
showing that $\sqrt[m]{(-1)^{(m+d)/m}c\alpha}$ is a product of $d/m$ elements of $B$. If $\beta$ is another element of $B \setminus \{0, \infty\}$, then we have a similar expression for $\beta$ as \eqref{prod2}, which yields
\begin{equation} \label{quotient}
\sqrt[m]{\frac{\alpha}{\beta}} = \frac{\sqrt[m]{(-1)^{(m+d)/m}c\alpha}}{\sqrt[m]{(-1)^{(m+d)/m}c\beta}} \in R.
\end{equation}

From \eqref{quotient} and the definition of $R$, it follows that for every $r \in R$ there is $s \in R$ with $s^m = r$.  To complete the proof, it is sufficient to show that $\zeta \in R$. Indeed, we show that $\zeta$ is a quotient of elements of $\phi^{-1}(b)$. First note that $\phi^{-1}(b) \setminus \{0, \infty\}$ is non-empty, for otherwise the assumptions $\phi(0) = 0$ and $\phi(\infty) = \infty$ imply that $b \in \{0, \infty\}$ and $\phi^{-1}(b) = \{b\}$.
Thus we take $\alpha \in \phi^{-1}(b) \setminus \{0, \infty\}$. Because $\phi \in K(x^m)$, we have $\phi(\zeta \alpha) = \phi(\alpha) = b$. Therefore $\zeta = \zeta \alpha / \alpha$, as desired. 
\end{proof}

To streamline the statement of our next result, we recall that a set $E \subset \mathbb{P}^1(\overline{K})$ is an \textit{exceptional set for $\phi$} if $E$ is finite and $\phi^{-1}(E) = E$. We further say that $b \in \mathbb{P}^1(\overline{K})$ is exceptional for $\phi$ if $b$ belongs to an exceptional set for $\phi$ \label{exceptionaldef}. Note that if $b$ is an exceptional point for $\phi$, then $K(\phi^{-\infty}(b))$ is a finite extension of $K$. Exceptional points are rare for maps defined over $\C$: by \cite[Theorem 1.6]{jhsdynam}, if $E$ is a non-empty exceptional set for $\phi \in \C(x)$ then $E$ has at most two elements, and if $\#E = 1$ then $\phi$ is conjugate over $\overline{K}$ to a polynomial, while if $\#E = 2$ then $\phi$ is conjugate over $\overline{K}$ to $x^{-d}$.

\begin{corollary} \label{deep periodic}
Let $K$ be a field and $\phi \in K(x)$ a rational function of degree $d \geq 2$, where $d$ is not divisible by the characteristic of $K$. Suppose that $\phi \in K(x^m)$ for some integer $m \geq 2$, and that $0$ and $\infty$ are periodic under $\phi$. Then $K(\mum) \subseteq K(\phi^{-\infty}(b))$ unless $b \in \{0, \infty\}$ and $b$ is exceptional for $\phi$.
\end{corollary}

\begin{proof}
The hypotheses imply that there is $n \geq 1$ with $\phi^n(0) = 0$ and $\phi^n(\infty) = \infty$. Moreover, $\phi \in K(x^m)$ implies that $\phi^n \in K(x^m)$.   Theorem \ref{rootsof1} then gives $K(\mum) \subseteq K(\phi^{-\infty}(b))$ unless $b \in \{0, \infty\}$ and $(\phi^n)^{-1}(b) = \{b\}$. 
This last equality implies that
$$
E = \{b, \phi(b), \phi^2(b), \ldots, \phi^{n-1}(b)\}
$$
is an exceptional set for $\phi$ (cf. the proof of \cite[Theorem 1.7]{jhsdynam}), whence $b$ is exceptional for $\phi$. 
\end{proof}

We can say more precisely when various roots of unity appear in the tower of fields $K(\phi^{-n}(b))$. Under the hypotheses of Theorem \ref{rootsof1}, it follows from the proof of that theorem that $\zeta_{m^j} \in K(\phi^{-j}(b))$ for each $j \geq 1$. Thus under the hypotheses of Corollary \ref{deep periodic} we have that 
\begin{equation} \label{rootsappear}
 \zeta_{m^j} \in K(\phi^{-rj}(b)),   
\end{equation} 
where $r$ is the least common multiple of the periods of $0$ and $\infty$. As an illustration, in \cite[Lemma 1.4]{arithbas} it's shown that for $K = \Q$, $\phi(x) = x^2 - 1$, and $b \in \Q \setminus \{-1, 0\}$, 
$$\zeta_{2^j} \in K(\phi^{-(2j-1)}(b)).$$
This is consistent with \eqref{rootsappear}, though it shows that the quantity $-rj$ is not necessarily minimal. 

We now give a result about roots of unity and preimage fields for maps that are only conjugate to those where Corollary \ref{deep periodic} applies. As a preliminary observation, note that if $\mu \in K(x)$ is a M\"obius transformation and $\phi \in K(x)$, and we set $\psi = \mu \circ \phi \circ \mu^{-1}$, then $\phi^n(\alpha) = b$ if and only if $\psi^n(\mu(\alpha)) = \mu(b)$. It follows that 
\begin{equation} \label{conjeq}
K(\phi^{-\infty}(b)) = K(\psi^{-\infty}(\mu(b)))
\end{equation}
Observe that \eqref{conjeq} only holds under the assumption that $\mu$ is defined over the ground field $K$. If the coefficients of $\mu$ lie in a non-trivial extension of $K$, then $K(\phi^{-\infty}(b))$ and $K(\psi^{-\infty}(\mu(b)))$ are not necessarily identical extensions of $K$; for instance in this case there is no guarantee that $\mu(b) \in K(\phi^{-\infty}(b))$. In this case we must enlarge $K$ to include the coefficients of $\mu$ in order to obtain \eqref{conjeq}. That is the approach taken in the proof of the next result.

\begin{corollary} \label{conjunity}
Let $K$ be a field and $\phi \in K(x)$ a rational function of degree $d \geq 2$, where $d$ is not divisible by the characteristic of $K$. Suppose that there are periodic points $\alpha$ and $\beta$ in $\mathbb{P}^1(\overline{K})$ and $m \geq 2$ such that 
\begin{equation} \label{conjhyp}
\mu \circ \phi \circ \mu^{-1} \in \overline{K}(x^m),
\end{equation}
where $\mu$ is a M\"obius transformation defined over $K(\alpha, \beta)$ with $\mu(\alpha) = 0$ and $\mu(\beta) = \infty$. Assume that $b$ is not exceptional for $\phi$. Then the compositum of $K(\alpha, \beta)$ and $K(\phi^{-\infty}(b))$ contains $K(\mum)$. 
\end{corollary}

\begin{proof}
Let $F = K(\alpha, \beta)$ and $\psi = \mu \circ \phi \circ \mu^{-1}$.  Because both $\psi$ and $\mu$ are defined over $F$, we have from \eqref{conjeq} that  $F(\phi^{-\infty}(b)) = F(\psi^{-\infty}(\mu(b))$.
Now $0$ and $\infty$ are periodic points for $\psi$, and together with \eqref{conjhyp} this means we may apply Corollary \ref{deep periodic} to $\psi$. Hence $F(\mum) \subseteq F(\psi^{-\infty}(\mu(b)) = F(\phi^{-\infty}(b))$, unless $\mu(b)$ is exceptional for $\psi$. Being an exceptional point is preserved under conjugation, so $\mu(b)$ is exceptional for $\psi$ if and only if $b$ is exceptional for $\phi$, contrary to hypothesis. 
Therefore Corollary \ref{deep periodic} gives $F(\mum) \subseteq F(\phi^{-\infty}(b))$, as desired.
\end{proof}

We now give an application to certain bicritical rational maps, that is, those with precisely two critical points in $\mathbb{P}^1(\overline{K})$. Note that any polynomial with precisely one critical point in $\overline{K}$ (i.e. a unicritical polynomial) is also a bicritical rational map. 

\begin{corollary} \label{twocritpts}
Let $K$ be a field and $\phi \in K(x)$ a rational function of degree $d \geq 2$, where $d$ is not divisible by the characteristic of $K$. Suppose that $\phi$ is bicritical with critical points $\alpha$ and $\beta$, both of which are periodic under $\phi$. Assume that $b$ is not exceptional for $\phi$. Then the compositum of $K(\alpha, \beta)$ and $K(\phi^{-\infty}(b))$ contains $K(\mud)$. \end{corollary} 

\begin{proof}
Let $\alpha$ and $\beta$ be the two critical points of $\phi$, which are by hypothesis periodic. By the Riemann-Hurwitz theorem, both $\alpha$ and $\beta$ must have ramification index $d$. We may apply a conjugation $\mu$ defined over $K(\alpha, \beta)$ that takes $\alpha$ to $0$ and $\beta$ to $\infty$, as in Corollary \ref{conjunity}, and take $\psi = \mu \circ \phi \circ \mu^{-1}$. Both $0$ and $\infty$ are points of ramification index $d$ for $\psi$, and it follows that $\psi$ has the form $(c_1x^d + c_2)/(c_3x^d + c_4)$ for some $c_1, c_2, c_3, c_4 \in K(\alpha, \beta)$. An application of Corollary \ref{conjunity} completes the proof.
\end{proof}

\begin{example}[Quadratic maps] Let $K$ be a field of characteristic different from 2. If $\phi \in K(x)$ has degree 2, then it is automatically bicritical. Let $\alpha, \beta \in \mathbb{P}^1(\overline{K})$ be the critical points of $\phi$. If both $\alpha$ and $\beta$ are periodic, then Corollary \ref{twocritpts} shows that $K(\mutwo)$ lies in the compositum of $K(\alpha, \beta)$ and $K(\phi^{-\infty}(b))$. The case of $K = \Q$ and $\phi(x) = x^2 - 1$, studied in \cite[Section 1]{arithbas}, is an example of this, with its two critical points $\infty$ and $0$ being periodic with periods one and two, respectively. 

For many quadratic rational functions with periodic critical points, we have $K(\alpha, \beta) = K$. If either $\alpha$ or $\beta$ is $\infty$, then the other lies in $K$ and $K(\alpha, \beta) = K$ holds. If $\{\alpha, \beta\} \cap \{\infty\} = \emptyset$, then $\alpha$ and $\beta$ are roots of a quadratic polynomial $K[x]$, and thus either lie in $K$ or are Galois conjugate over $K$. In the latter case, the Galois conjugacy extends to the orbits of $\alpha$ and $\beta$, and in particular both $\alpha$ and $\beta$ must have the same period. 
\end{example}

\subsection{Chebyshev and Latt\`es maps} \label{special}

Let $K$ be a field and $d \geq 2$ an integer not divisible by the characteristic of $K$. Recall that the Chebyshev polynomial $T_d$ of degree $d$ is the map given by the commutative diagram
\begin{equation} \label{comm}
\begin{tikzpicture}[baseline=(current  bounding  box.center)]
    \node (LL) at (0,0) {$\mathbb{P}^1$};
    \node (UR) at (3,2) {$\mathbb{P}^1$};
    \node (UL) at (0,2) {$\mathbb{P}^1$};
    \node (LR) at (3,0) {$\mathbb{P}^1$};
    \draw[->] (UL)--(UR) node [pos=0.5, above,inner sep=0.2cm] {$x^d$};
    \draw[->] (UL)--(LL) node [pos=0.4, left, outer sep=0.1cm] {$\theta$};
    \draw[->] (UR)--(LR) node [pos=0.4, right, outer sep=0.1cm] {$\theta$};
    \draw[->] (LL)--(LR) node [pos=0.5, above,inner sep=0.2cm] {$T_d$};
\end{tikzpicture}
\end{equation}
where $\theta(x) = x + 1/x$. Thus $T_d$ is the action of the endomorphism $x^d$ of the multiplicative group after quotienting by the automorphism $x \mapsto 1/x$. Note that the reason that $x^d$ descends to a well-defined map in the diagram above is that it commutes with $x \mapsto 1/x$, and hence maps any fiber $\{\alpha, 1/\alpha\}$ of $\theta$ to another fiber, namely $\{\alpha^d, 1/\alpha^d\}$.

We denote by $K(\zeta_{m^\infty} + \zeta_{m^\infty}^{-1})$ the extension of $K$ generated by adjoining all elements of $\overline{K}$ of the form $\zeta + \zeta^{-1}$, where $\zeta^{m^n} = 1$ for some $n \geq 1$. We generalize a result due to Gottesman and Tang in the case $d = 2$ \cite{richtang}: 

\begin{theorem} \label{chebmain}
Let $K$ a field, let $d \geq 2$ be an integer not dividing the characteristic of $K$, and let $b \in K$. 
Then $K(\zeta_{d^\infty} + \zeta_{d^\infty}^{-1}) \subseteq K(T_d^{-\infty}(b))$. In particular, there is a degree-2 extension of $K(T_d^{-\infty}(b))$ that contains $K(\mud)$.
\end{theorem}

\begin{proof}
Fix an integer $n \geq 1$ and let $b' \in \overline{K} \setminus \{0\}$ satisfy $\theta(b') = b$. Let $P_{d}(x) = x^{d}$, take $\gamma \in P_{d}^{-n}(b')$, and let $\zeta$ be a primitive $d^n$-th root of unity. Note that 
\begin{equation} \label{chebfiber}
P_{d}^{-n}(b') = \{\zeta^i \gamma : 0 \leq i \leq d^n-1\}.
\end{equation}
Observe that the commutativity of \eqref{comm} gives $\theta(P_{d}^{-n}(b')) = T_{d}^{-n}(b)$.

We claim that $\theta(\zeta) \in K(T_{d}^{-n}(b))$, which proves the first assertion of the theorem.
The function $\theta$ satisfies the identity
\begin{equation} \label{identity}
\theta(x)\theta(y) = \theta(xy) + \theta(x/y).
\end{equation}
Take $x = \gamma$ and $y = \zeta$ in \eqref{identity} to obtain
\begin{equation} \label{goodquot}
\theta(\zeta) = \frac{\theta(\zeta \gamma) + \theta(\zeta^{-1} \gamma)}{\theta(\gamma)} \in K(\theta(P_{d^n}^{-1}(b'))) = K(T_{d}^{-n}(b)).
\end{equation}
To prove the last statement of the theorem, observe that $\zeta_{d^n}$ is a root of 
$$x^2 - (\zeta_{d^n} + \zeta_{d^n}^{-1})x + 1 \in K(\zeta_{d^\infty} + \zeta_{d^\infty}^{-1})[x].$$
Thus $L_n := K(T_d^{-\infty}(b))(\zeta_{d^n})$ has degree at most two over $K(T_d^{-\infty}(b))$. Note also that $L_n \subseteq L_{n+1}$ for $n \geq 1$. It follows that $L_\infty := \cup_{n \geq 1} L_n$ has degree at most two over $K(T_d^{-\infty}(b))$, as desired.
\end{proof}

\begin{remark}
In the case where $\phi = \mu \circ T_d \circ \mu^{-1}$ for some M\"obius transformation $\mu \in K'(x)$, where $K'$ is a finite extension of $K$, it follows from \eqref{conjeq} and Theorem \ref{chebmain} that $K'(\zeta_{d^\infty} + \zeta_{d^\infty}^{-1}) \subseteq K'(\phi^{-\infty}(b))$ provided that $b \neq \mu(\infty)$. In this case $K(\mud)$ is contained in a degree-2 extension of $K'(\phi^{-\infty}(b))$, and hence in a finite extension of $K(\phi^{-\infty}(b))$.
\end{remark}

\begin{remark}
One might hope to obtain the results of Theorem \ref{chebmain} in some cases where $\phi$ merely factors through a Chebyshev polynomial, i.e., $\phi = \psi(T_d(x))$ for $\psi \in K(x)$ of degree at least two meeting certain hypotheses. Let $b' \in \theta^{-1}(b)$, and set $K' = K(b')$. In the situation where there is a rational function $\ell_\psi \in K(x)$ with $\theta \circ \ell_\psi = \psi \circ \theta$, and moreover $\ell(0) = 0$ and $\ell(\infty) = \infty$, it follows from Theorem \ref{rootsof1} that $K'(\mud) \subset K'((\ell_\psi \circ P_d)^{-\infty}, b')$. Indeed, from the proof of Theorem \ref{rootsof1} we have that a primitive $d^n$-th root of unity $\zeta$ is a quotient of two products of equal numbers of elements of $(\ell_\psi \circ P_d)^{-n}(b')$. However, there does not seem to be a way to express the image of this quotient under $\theta$ as a rational function of elements of $\theta((\ell_\psi \circ P_d)^{-\infty}(b'))$, as in \eqref{goodquot}.
\end{remark}

We turn now to Latt\`es maps, i.e. those $\phi \in K(x)$ arising from a commutative diagram
\begin{equation} \label{Lattespic}
\begin{tikzpicture}[baseline=(current  bounding  box.center)]
    \node (LL) at (0,0) {$\mathbb{P}^1$};
    \node (UR) at (3,2) {$E$};
    \node (UL) at (0,2) {$E$};
    \node (LR) at (3,0) {$\mathbb{P}^1$};
    \draw[->] (UL)--(UR) node [pos=0.5, above,inner sep=0.2cm] {$\psi$};
    \draw[->] (UL)--(LL) node [pos=0.4, left, outer sep=0.1cm] {$\pi$};
    \draw[->] (UR)--(LR) node [pos=0.4, right, outer sep=0.1cm] {$\pi$};
    \draw[->] (LL)--(LR) node [pos=0.5, above,inner sep=0.2cm] {$\phi$};
\end{tikzpicture}
\end{equation}
where $E$ is an elliptic curve, $\psi$ is a morphism, and $\pi$ is a finite separable covering. We assume that $\pi$ is defined over $K$, but allow $\psi$ a priori to be defined over a finite extension of $K$. Here we focus on \textit{flexible Latt\'es maps}, namely those for which $\psi(P) = [d]P + Q$ for some integer $d \geq 2$ and some $Q \in E(\overline{K})$, and $\deg(\pi) = 2$ with $\pi(P) = \pi(-P)$ for all $P \in E(\overline{K})$. In this case $\deg \phi = d^2$.
Although we will not need it here, one can show \cite[Proposition 6.51]{jhsdynam} that in fact $Q$ is a 2-torsion point for $E$, and that after replacing $\phi$ by a conjugate we may assume that $\pi(x,y) = x$. 

\begin{theorem} \label{Lattesroots}
Let $K$ be a field, let $d \geq 2$ be an integer not dividing the characteristic of $K$, and let $b \in \mathbb{P}^1(K)$. If $\phi \in K(x)$ is a flexible Latt\`es map of degree $d^2$, then there exists an extension $F$ of $K(\phi^{-\infty}(b))$ with $K(\mud) \subseteq F$ and $[F : K(\phi^{-\infty}(b))] \leq 8$.
\end{theorem}

\begin{proof}
Let the map associated to $\phi$ in \eqref{Lattespic} be $\psi(P) = [d]P + Q$.
Fix an integer $n \geq 1$ and let $B \in E(\overline{K})$ satisfy $\pi(B) = b$. Take $B_n \in \psi^{-n}(B)$, and note that $\psi^n$ has degree $d^{2n}$ and maps $P$ to $[d^n]P$ plus a multiple of $Q$. It follows that 
\begin{equation} \label{fiber}
\psi^{-n}(B) = \{B_n + T : T \in E[d^n]\}.
\end{equation}
Therefore $E[d^n] \subset K(\psi^{-n}(B))$, and by properties of the Weil pairing \cite[Corollary 8.1.1]{AEC} we have $K(\zeta) \subset K(\psi^{-n}(B))$, where $\zeta$ is a primitive $d^n$th root of unity. Note that \eqref{fiber} also implies $K(\psi^{-n}(B)) = K(B_n, B_n + T_1, B_n + T_2)$, where $T_1$ and $T_2$ are generators of $E[d^n]$. 

Now the commutativity of \eqref{Lattespic} gives $\pi(\psi^{-n}(B)) = \phi^{-n}(b)$. 
Because $\pi$ is defined over $K$, this gives that $K(\phi^{-n}(b))$ is a subfield of $K(\psi^{-n}(B)),$ whence 
\begin{equation} \label{chain}
K(\pi(B_n), \pi(B_n + T_1), \pi(B_n + T_2)) \subseteq K(\pi(\psi^{-n}(B))) = K(\phi^{-n}(b)) \subseteq K(\psi^{-n}(B)).
\end{equation}
Because $\deg \pi = 2$, the last extension in \eqref{chain} has degree at most 8 over the first extension, and hence $[K(\psi^{-n}(B)) : K(\phi^{-n}(b))] \leq 8$.

To complete the proof, take 
$
F = K(\psi^{-\infty}(B)).
$
As in the first paragraph of this proof, we have $E[d^\infty] \subseteq F$, and by the existence of the Weil pairing it follows that $K(\mud) \subseteq F$. If $\gamma \in F$, then there is $n \geq 1$ with $\gamma \in K(\psi^{-n}(B)),$ and  \eqref{chain} implies that 
\begin{equation} \label{eight}
[K(\phi^{-\infty}(b))(\gamma) : K(\phi^{-\infty}(b))] \leq 8.
\end{equation}
Because $\pi$ is separable, we have that $K(\psi^{-\infty}(B))$ is a separable extension of $K(\phi^{-\infty}(b))$. The theorem now follows from \eqref{eight} and the primitive element theorem once we establish that 
$[K(\psi^{-\infty}(B)) : K(\phi^{-\infty}(b))]$ is finite. To see this, we may replace $K$ by a finite extension if necessary and assume that $Q$ is defined over $K$. Then $K(\psi^{-n}(B)) \subseteq K(\psi^{-(n+1)}(B))$ for all $n \geq 1$, and because each of these extensions has degree at most 8 over $K(\phi^{-\infty}(b))$, it follows that their union, which is $K(\psi^{-\infty}(B))$, does as well. 
\end{proof}

In some circumstances the bound of 8 in Theorem \ref{Lattesroots} can be improved. If $b = \infty$, then after replacing $\phi$ by a conjugate, we may assume that  $K(\phi^{-n}(b)) = K(x(E[d^n]))$, the extension of $K$ generated by the $x$-coordinates of the $d^n$-torsion points of $E$. The group law on $E$ then implies that $[K(E[d^n]) : K(x(E[d^n]))] \leq 2$ (see e.g. \cite[Lemma 2.2]{bandini}). Thus the bound is 2 in this case.

\section{Results on higher ramification groups} \label{preliminaries}

In this section we deduce Theorem \ref{mainram1} from Theorem \ref{mainroots}. We first give some preliminary results on higher ramification groups. 

\begin{lemma} \label{tower}
Let $K$ be a finite extension of $\Q_p$, and $L$ a Galois extension of $K$. Then $\Gal(L/K)^u$ is infinite for all $u \in \mathbb{R}^+$ if and only if there is a sequence $(F_n)_{n \geq 0}$ of finite Galois sub-extensions of $L$ with 
$$K := F_0 \subset F_1 \subset F_2 \subset \cdots$$
such that for each $u \in \R^+$ and $B > 0$ there exists $n \geq 0$ with $|\Gal(F_n/K)^u| > B$.
\end{lemma}

\begin{proof}
This follows from the definition of the upper numbering for infinite extensions and Herbrand's theorem on ramification (see e.g. \cite[Proposition IV.14]{serrebook}) which gives that $G^u$ is infinite if $(G/H)^u$ is infinite.
\end{proof}

\begin{lemma} \label{deepramground}
Let $K$ be a finite extension of $\Q_p$, $L$ an arbitrary Galois extension of $K$, and $K'$ a finite Galois extension of $K$. Then the following are equivalent:
\begin{enumerate}
    \item $\Gal(L/K)^u$ is infinite for all $u \in \mathbb{R}^+$.
    \item $\Gal(LK'/K)^u$ is infinite for all $u \in \mathbb{R}^+$;
    \item $\Gal(LK'/K')^u$ is infinite for all $u \in \mathbb{R}^+$;
\end{enumerate}
\end{lemma}

\begin{proof}
We have the following diagram:
$$ \begin{tikzpicture}

    \node (K) at (0,0) {$K$};
    \node (K') at (2,1.5) {$K'$};
    \node (LK') at (0,3) {$LK'$};
    \node (L) at (-2,1.5) {$L$};

    \draw (K)--(K'); 
    \draw (K)--(L) node [pos=0.7, below,inner sep=0.4cm] {$R$};
    \draw (LK')--(L) node [pos=0.7, above,inner sep=0.4cm] {$H$};
    \draw (LK')--(K') node [pos=0.7, above, inner sep=0.4cm] {$P$};
    \draw (K)--(LK') node [pos=0.5, right, inner sep=0.2cm] {$G$};

    \end{tikzpicture}
    $$
Here $G = \Gal(LK'/K)$, $H = \Gal(LK'/L)$, $P = \Gal(LK'/K')$, and $R = \Gal(L/K)$.  
Note that $P \triangleleft G$, $H \triangleleft G$, and $G/H \simeq R$.  Recall that by Herbrand's theorem, $(G^uH)/H \simeq (G/H)^u$. If $G^u$ is infinite for all $u$, then since $|H| = [K' : (L \cap K')] \leq [K':K] < \infty$, we have that $(G/H)^u \simeq R^u$ is infinite for all $u$, so $(2) \Rightarrow (1)$.  Similarly, (1) $\Rightarrow$ (2) directly from Herbrand's theorem.

Next, assume $G^u$ is infinite for all $n$.  Since $[K':K]$ is finite, so is $G/P$ and therefore for all $u$ so is $G^u/(G^u \cap P)$.  It follows that $G^u \cap P$ is infinite for all $u$.  From \cite[Lemma 5.5]{Poonen2}, we have $G^u \cap P \leq P^u$, so $P^u$ is infinite for all $u$.  Therefore (2) $\Rightarrow$ (3).

Finally, assume $P^u$ is infinite for all $u$.  
Let $B > 0$ be fixed. By Lemma \ref{tower} there is a sequence ${F_n}$ of finite Galois sub-extensions of $LK'/K'$ such that for each $u \in \Z^+$, there exists $m$ such that $|\Gal(F_m/K')^{u}|>B$.
Let $X(n)=\Gal(F_n/K)$, and let $Y(n) = \Gal(F_n/K')$; we then have, for all $v \in \Z^+$,
\begin{align*}
    [X(n):X(n)_v] &= [X(n):Y(n)X(n)_v][Y(n)X(n)_v:X(n)_v] \\
    & \geq [Y(n)X(n)_v:X(n)_v] = [Y(n):Y(n)_v]
\end{align*}

If $\phi$ is the Herbrand transition function for the specified extension, we then have
\begin{equation} \label{transition}
\phi_{F_n/K}(x) = \int^x_0 \frac{dt}{[X(n)_0:X(n)_t]} 
\leq  \int^x_0  \frac{dt}{[Y(n)_0:Y(n)_t]} = \phi_{F_n/K'}(x).
\end{equation}

Therefore, if $v = \phi^{-1}_{F_n/K'}(u)$, we have \begin{align*}
    |X(n)^u| & = |X(n)^{\phi_{F_n/K'}(v)}| = |X(n)_v| \\
     & \geq |Y(n)_v| = |Y(n)^{\phi_{F_n/K}(v)}|\\
     & \geq |Y(n)^{\phi_{F_n/K'}(v)}| = |Y(n)^u|
\end{align*}

So if $|Y(m)^u| > B$, we also have $|X(m)^u| > B$;
 hence by Lemma \ref{tower} we have $G^u$ infinite for all $u$, as desired.
\end{proof}

\begin{lemma}
\label{unitylem}
Let $K$ be a finite extension of $\Q_p$ and $L$ a Galois extension of $K$. Suppose that $E$ is a finite extension of $L$ with $K(\mum) \subseteq E$ for some multiple $m$ of $p$. Then $\Gal(L/K)^u$ is infinite for all $u \in \mathbb{R}^+$. 
\end{lemma}

\begin{proof}
From the primitive element theorem we have $E = L(\alpha)$ for some $\alpha \in \overline{K}$. Let $K'$ be the Galois closure of $K(\alpha)$ over $K$, which is a finite extension of $K$.  Then $LK'$ is Galois over $K$ (as a compositum of a Galois extension with a finite Galois extension), and is a finite extension of $E$ since $\Gal(LK'/L) \simeq \Gal(K'/K)$.  We can then assume $E$ is Galois over $K$ and $L$ by replacing $E$ with $LK'$ if necessary.

From \cite[Proposition IV.18]{serrebook} we have that $\Gal(K(\mup)/K)^u$ is infinite for all $u \in \mathbb{R}^+$. Because $p \mid m$, we have $K(\mup) \subseteq E$, and Herbrand's theorem immediately implies that $\Gal(E/K)^u$ is infinite for all $u \in \R^+$.  An application of Lemma \ref{deepramground} then completes the proof. 
\end{proof}

\begin{proof}[Proof of Theorem \ref{mainram1}]
The theorem follows immediately from Theorem \ref{mainroots} and Lemma \ref{unitylem}.
\end{proof}

\section{Some iterated extensions that are branch-APF} \label{apf}

\label{apfdisc}
Recall from the introduction that if $K$ is a finite extension of $\Q_p$, an infinite, algebraic, totally wildly ramified extension $L$ of $K$ is APF if 
$[G_K : G_K^uG_L] < \infty$ for all $u \in \mathbb{R}^+$, where $G_K$ denotes the absolute Galois group $\Gal(\overline{K}/K)$ and similarly for $G_L$. In this section we use the main result of \cite{CDL}, which characterizes arithmetically profinite extensions, to prove Theorem \ref{mainapfintro}, which shows that $K(\phi^{-\infty}(b))$ is branch-APF for certain $\phi$ and $b$. In particular, this shows that $\Gal(K(\phi^{-\infty}(b)))^u$ is infinite for all $u \in \mathbb{R}^+$. Indeed,
let $E$ be a Galois extension of $K$ containing an APF extension $L/K$. Because $L/K$ is infinite,  $[G_K^uG_L : G_L]$ must be infinite as well. However, one verifies that
$$
[G_K^uG_L : G_L] \leq [G_K^uG_E : G_E],
$$
and by Herbrand's theorem on ramification, 
the latter is $(G_K/G_E)^u$, i.e. $\Gal(E/K)^u$.

 A motivating application of \cite{CDL} is dynamical, though the basepoint is required to be a uniformizer. Sing \cite{sing1}, by skillfully employing the concrete description of the Hasse-Herbrand function given in \cite{lubin}, furnishes a highly detailed description of the higher ramification filtration arising from certain post-critically bounded polynomials, independent of basepoint. In our Theorem \ref{mainapfintro} we content ourselves to prove branch-APFness, for a broader class of $\phi$ than in either \cite{CDL} or \cite{sing1}, but with some restrictions on the basepoint (though in general for more basepoints than just uniformizers, as in the examples given in \cite{CDL}).

Throughout this section, we assume that a finite extension $K$ of $\Q_p$ has maximal ideal $\p$, ring of integers $\O_K$, and valuation $v_K$ normalized so that $v_K(\O_K) = \Z$.
For our purposes, we require only one direction of \cite[Theorem 1.1]{CDL}, which we state for convenience:
\begin{theorem}[\cite{CDL}] \label{maintool}
Let $E$ be a finite extension of $\Q_p$, and let $L$ be an infinite, totally ramified extension of $E$. Then $L/E$ is (strictly\footnote{For a definition of {\em strictly} APF, see Section 1.4.1 of \cite{Wintenberger} or Remark 2.8 of \cite{CDL}. In our constructions, all APF extensions are strictly APF.}) APF provided that there exists a tower $\{E_n\}_{n \geq 2}$ of finite extensions of $E_1:=E$ inside $L$ with $L = \cup E_n$ and a norm-compatible sequence $\{\pi_n\}_{n \geq 2}$ with $\pi_n$ a uniformizer of $E_n$ such that:
\begin{enumerate}
    \item The degrees $q_n := [E_{n+1} : E_n]$ are bounded above.
    \item If $g_n(x) = x^{q_n} + a_{n,q_{n-1}}x^{q_n-1} + \dots + a_{n,1}x + (-1)^p\pi_n \in E_n[x]$ is
the minimal polynomial of $\pi_{n+1}$ over $E_n$, then the non-constant and non-leading coefficients $a_{n,i}$ of $g_n$ satisfy $v_E(a_{n,i}) > \epsilon$ for some $\epsilon > 0$, independent of $n$ and $i$.
\end{enumerate}
\end{theorem}

In our application, the coefficients $a_{n,i}$ are independent of $n$.

The authors of \cite{CDL} note that their theorem is perhaps better suited to constructing APF extensions that to verifying that a given extension is APF. A particular point of difficulty comes in finding the norm-compatible sequence $\pi_n$. To realize our desired application, we locate this sequence in terms of iterated preimages, and this constrains the basepoints to which our theorem applies.  

Before proving Theorem \ref{mainapfintro}, note that for $\phi \in K(x)$, we may assume that all coefficients lie in $\O_K$. We denote the  coefficient-wise reduction modulo $\p$ of $\phi$ as $\overline{\phi}(x) \in (\O_K / \p)(x)$. Recall that $\phi$ has good reduction if $\deg \phi = \deg \overline{\phi}$. For convenience we restate Theorem \ref{mainapfintro}: 

\begin{theorem} \label{mainapf}
Let $K$ be a finite extension of $\Q_p$. Suppose that $\phi \in K(x)$ has good reduction and that there exists $m \geq 1$ with 
\begin{equation} \label{redhyp}
\overline{\phi^m}(x) = cx^{p^r} 
\end{equation}
for some $c \neq 0$ and $r \geq 1$. Then $\phi^m$ has a unique fixed point $\gamma$ (resp. $\delta$) with positive (resp. negative) $p$-adic valuation. 
Additionally, there is a unit $w \in \overline{K}$ such that $K(\phi^{-\infty}(b))$ is branch-APF over $K(\gamma, \delta, w)$, where 
$$
b = \frac{\pi + w\gamma}{\delta^{-1}\pi + w}
$$
for any uniformizer $\pi$ of $K(\gamma, \delta, w)$.
\end{theorem}

\begin{proof}
Our goal here is to apply Theorem 4.1.  To do so, we need to demonstrate a tower $\{E_n\}$ of subextensions of $K(\phi^{-\infty}(b))$ with a norm-compatible sequence of uniformizers $\{\pi_n\}$ satisfying the two conditions of Theorem 4.1.  We define this tower and sequence of uniformizers via the non-unit fixed points of a conjugate of $\phi^m$, and show that the tower of extensions aligns with a subsequence of the extensions give by the preimages of $b$ under $\phi$.

Write $\phi^m(x) = f(x)/g(x)$ where $f, g \in \O_K[x]$. Because $\phi$ has good reduction, it follows from \eqref{redhyp} that
$f(x) = \sum_{i = 0}^{p^r} a_ix^i$ and $g(x) = \sum_{i = 0}^{p^r} b_ix^i$,
with $v_K(a_{p^r}) = v_K(b_0) = 0$ and all other coefficients having positive valuation.  
We then examine the Newton polygon of $f(x) - xg(x)$ to get information on the fixed points of $\phi^m(x)$.  If $f(x) - xg(x) = \sum_{i = 0}^{p^r+1} c_ix^i$, note that 
\begin{itemize}
    \item $v_K(c_{p^r+1}) = v_K(b_{p^r}) > 0$,
    \item $v_K(c_{p^r}) = v_K(a_{p^r} - b_{p^r-1}) = v_K(a_{p^r}) = 0$,
    \item $v_K(c_1) = v_K(a_1-b_0) = v_K(b_0) = 0$, 
    \item and $v_K(c_0) = v_K(a_0) > 0$.
\end{itemize}  
We therefore have a unique fixed point $\delta$ of $\phi^m$ with $v_K(\delta) < 0$ and a unique fixed point $\gamma$ of $\phi^m$ with $v_K(\gamma) > 0$. 
Put 
\begin{equation} \label{newmudef}
\mu(x) = w \cdot \left(\frac{x - \gamma}{1 -\delta^{-1}x} \right),
\end{equation}
where $w \in \overline{K}$ is a unit that will be specified later, and set $\psi = \mu \circ \phi \circ \mu^{-1}$. Observe that $\overline{\mu}(x) = \overline{w}x$, where $\overline{w} \neq 0$, and thus from \eqref{redhyp} we have that $\psi$ has good reduction and 
\begin{equation} \label{redhyp2}
\text{$\overline{\psi^m}(x) = c_1x^{p^r}$ for $c_1 = c\overline{w}^{p^r - 1} \neq 0$.}
\end{equation}
Write $\psi^m(x) = f_1(x)/g_1(x)$ where $f_1, g_1 \in \O_K[x]$ with
\begin{equation} \label{st}
f_1(x) = \sum_{i = 0}^{p^r} s_ix^i \qquad \text{and} \qquad g_1(x) = \sum_{i = 0}^{p^r} t_ix^i.
\end{equation}
The conjugation by $\mu$ yields $\psi^m(0) = 0$ and $\psi^m(\infty) = \infty$, and thus $s_0 = t_{p^r} = 0$. 
Moreover, \eqref{redhyp2} shows that $s_{p^r}$ and $t_0$ are both units, and also that $\overline{s_{p^r}} / \overline{t_0} = c\overline{w}^{p^r - 1}$. We now select $w$  so that $s_{p^r} = t_0$ (note that $\overline{w} = c^{-1/(p^r - 1)}$ but that $w$ itself is a root of a polynomial with coefficients in $K(\gamma, \delta)$). 
Dividing through by this common value, we may assume that $s_{p^r} = t_0 = 1$. Note that in the case where $\delta = \infty$, we have
$
w^{p^r-1} = a_{p^r}/g(\gamma), 
$ and hence we may take $w = 1$ when $\phi$ is a monic polynomial. 

We now construct the norm-compatible sequence of uniformizers necessary to apply Theorem \ref{maintool}, which the special form of $\psi^m$ allows us to do using iterated preimages.
Let $\pi_1$ be a uniformizer for the field $E_1 := K(\gamma, \delta, w)$. For $n \geq 2$, inductively define the tuple $(h_n, \pi_n, E_n)$ such that $\pi_n$ is a root of the polynomial 
$$h_n(x) := f_1((-1)^{p+1}x) + (-1)^{p} \pi_{n-1}g_1((-1)^{p+1}x)$$
and $E_n = E_{n-1}(\pi_n)$. Observe that $h_n$ has degree $p^r$ and leading coefficient $s_{p^r} + (-1)^p \pi_{n-1} t_{p^r} = 1$, and satisfies $h_n(0) = s_0 + (-1)^{p} \pi_{n-1}t_0 = (-1)^{p}\pi_{n-1}$. Now each of $s_0, \ldots, s_{p^r-1}, t_1, \ldots, t_{p^r}$ has $v_K$ at least one, and thus $h_n$ has Newton polygon (with respect to the normalized valuation on $E_n$) consisting of a single segment of slope $-1/p^r$.
Hence $h_n$ is irreducible over $E_n$, $\pi_{n+1}$ is a uniformizer for $E_n(\pi_{n+1}) = E_{n+1}$, and 
$$N_{E_{n+1}/E_n}(\pi_{n+1}) = (-1)^{\deg h_n}(-1)^{p}\pi_{n} = \pi_n.$$
Therefore $\pi_n$ is a uniformizer for $E_n$, with minimal polynomial $h_n(x) \in E_{n-1}[x]$, and $N_{E_n/E_{n-1}}(\pi_n) = \pi_{n-1}$. 
Theorem \ref{maintool} then gives that $\cup E_n$ is APF over $E_1$.

Now $h_n(\pi_n) = 0$ implies that $\psi^m(\pi_n) = (-1)^{p+1}\pi_{n-1}$. When $p$ is odd, the sequence $\{\pi_n\}_{n \geq 1}$ shows that $L_\infty(\psi^m, \pi_1)$ is branch-APF over $E_1$. Define a sequence $\{\alpha_i\}_{i \geq 1}$ with $\alpha_{m(n-1) + 1} = \pi_n$ for all $n \geq 1$, and such that $\phi(\alpha_i) = \alpha_{i-1}$ for all $i \geq 1$. Then $E_1(\alpha_1, \alpha_2, \ldots) = E_1(\pi_1, \pi_2, \ldots)$, showing that $L_\infty(\psi, \pi_1)$ is branch-APF over $E_1$. From \eqref{conjeq} it follows that $L_\infty(\phi, \mu^{-1}(\pi_1))$ is branch-APF over $E_1$, and the theorem is proved with $\pi = \pi_1$ in the case of odd $p$.

If $p = 2$, then $h_n(\pi_n) = 0$ implies that $\psi^m(-\pi_n) = (-1)^{p+1}\pi_{n-1} = -\pi_{n-1}$. Therefore $\cup E_n$ is a branch sub-extension of $L_\infty(\psi^m, -\pi_1) $ over $E_1$, and as in the previous paragraph, $L_\infty(\phi, \mu^{-1}(-\pi_1))$ is branch-APF over $E_1$. As $\pi_1$ was an arbitrary uniformizer for $E_1$, $-\pi_1$ is also an arbitrary uniformizer, and the theorem is proved.
\end{proof}

\begin{example} \label{nonPCB}
Let $K = \Q_2$, and consider the polynomial $\phi(x) = (x-2)^8 - 2(x-2)^2 + 3$. Then $\phi$ has six critical points of valuation $-1/6$, and the orbit of each approaches $\infty$. But Theorem \ref{mainapf} applies because $\overline{\phi^2}(x) = x^{64}$. Indeed, we have $\delta = \infty$ and $w = 1$ because $\phi$ is a monic polynomial, and we also have $\gamma = 2$. Theorem \ref{mainapf} then gives that $L_\infty(\phi, \pi+2)$ is branch-APF over $\Q_2$ for any uniformizer $\pi$ of $\Q_2$. In particular, if $b \in 4\Z_2$ then taking $\pi = -2 + b$ shows that $K(\phi^{-\infty}(b))$ is branch-APF over $\Q_2$.
\end{example}

\section*{Acknowledgement}

The authors are grateful to the referees for their careful reading and many helpful suggestions.

\bibliographystyle{plain}

\end{document}